\documentclass[12pt]{amsart} 
\usepackage{amscd}
\usepackage{amssymb}
\usepackage{a4wide}
\usepackage{amstext}
\usepackage{amsthm}
\usepackage{xcolor}
\usepackage[T1,T2A]{fontenc}
\usepackage[utf8]{inputenc}
\usepackage[american]{babel}
\usepackage{url}
\usepackage{amsfonts}
\usepackage{amssymb, amsthm}
\usepackage{amsmath}
\usepackage{mathtools}
\usepackage{needspace}
\usepackage[pdftex]{graphicx}
\usepackage{hyperref}
\usepackage{datetime}
\usepackage{epigraph}
\usepackage{verbatim}
\usepackage{mathtools}
\usepackage{xcolor}
\usepackage{enumerate}
\usepackage[american]{babel}
\numberwithin{equation}{section}

\newcommand{\Z}{\mathbb{Z}}

\newcommand{\N}{\mathbb{N}}
\newcommand{\R}{\mathbb{R}}

\newcommand{\Cm}{\mathbb{C}}

\newcommand{\eps}{\varepsilon}

\newcommand{\F}{\mathcal{F}}
\newcommand{\G}{\mathcal{G}}

\DeclareMathOperator{\esssup}{ess\sup} 
\DeclareMathOperator{\essinf}{ess\inf}

\renewcommand{\phi}{\varphi}

\newtheorem{Thm}{Theorem}[section]
\newtheorem{theorem}[Thm]{Theorem}

\newtheorem{lemma}[Thm]{Lemma}

\newtheorem{corollary}[Thm]{Corollary}
\newtheorem{remark}[Thm]{Remark}

\newtheorem{definition}{Definition}

\begin{document}
\sloppy

\title[]{Openness of the frame set on the hyperbolas}
\author{Aleksei Kulikov}
\address{Tel Aviv University, School of Mathematical Sciences, Tel Aviv, 69978, Israel,
\newline {\tt lyosha.kulikov@mail.ru} 
}
\begin{abstract} { We prove that for the functions of the form $g(x) = h(x) + \frac{C}{x+i}$, where $h$ belongs to the continuous Wiener algebra $W_0$, the intersection of the frame set $\F_g$ with every hyperbola $\{\alpha, \beta > 0 \mid \alpha\beta = c\}$ is open in the relative topology. In particular, this applies to all rational functions $g$.
}
\end{abstract}
\maketitle
\section{Introduction}
Given a function $g\in L^2(\R)$ and numbers $\alpha, \beta > 0$, we consider the system of time-frequency shifts of the function $g$ with respect to the lattice $\alpha\Z\times \beta\Z$
$$\G(g;\alpha, \beta) = \{g_{n, m}(x) = g(x-n\alpha)e^{2\pi i m\beta x}\}_{n, m\in \Z}.$$

One of the main questions of the Gabor analysis is to determine when $\G(g;\alpha, \beta)$ is a frame, that is when there exist positive constants $A, B$ such that for all $f\in L^2(\R)$ we have
\begin{equation}\label{frame bound}
A||f||^2 \le \sum_{n, m\in\Z} |\langle f, g_{n, m}\rangle|^2\le B||f||^2.
\end{equation}

Usually we fix the function $g$ and ask for the characterization of its frame set $\F_g$, that is the set of pairs $(\alpha, \beta)$ such that  $\G(g;\alpha, \beta)$ is a frame. Over the years this was achieved for many different functions $g$  \cite{bs, DS, DZ, Jans2, Jans3, JansStr, L, S, SW}, as well as classes of functions $g$ \cite{bkl, Gro1, Gro2}. However, this list is still fairly restrictive and so there is an interest in the general properties of the frame set, under some hypothesis on the function $g$. 

Arguably the most important and well-known result of this sort is the density theorem which says that if $\alpha\beta > 1$ then  $\G(g;\alpha, \beta)$ is never a frame (see \cite{Jans5} for a simple self-contained proof). In fact, if $\alpha\beta > 1$ then the system $\G(g;\alpha, \beta)$ is not even complete in $L^2(\R)$ \cite{Rie}. If we know that the functions $xg(x)$ and $\xi\hat{g}(\xi)$ are also in $L^2(\R)$, then the Balian--Low theorem \cite{bal, low} says that any pair $(\alpha, \beta)$ with $\alpha\beta = 1$ also does not give us a frame. In general these are the only possible restrictions based only on the decay and smoothness of the function $g$, since when $g$ is a gaussian $e^{-\pi x^2}$ its frame set is all pairs $(\alpha, \beta)$ with $\alpha\beta < 1$ \cite{L, S, SW}. In the other direction, if the function $g$ is Schwartz then all pairs $(\alpha, \beta)$ with small enough $\alpha$ and $\beta$ necessarily belong to the frame set \cite{BC}.

Another type of results are the ones that establish the geometric properties of the frame set itself. If the function $g$ belongs to the Feichtinger algebra $M^1$  then $\F_g$ is an open subset of $\R^2$ \cite{Fei}. Belonging to the Feichtinger algebra $M^1$ or the continuous Wiener algebra $W_0$ are also the weakest known general conditions for which the right-hand estimate in \eqref{frame bound} holds. However, for many functions the frame set is the full set of pairs with $\alpha\beta \le 1$, such as $g(x) = \frac{1}{x+i}$ \cite{Jans2} and more generally $g(x)$ being a Herglotz rational function \cite{bkl}, which is not open. In this paper we show that for a wide class of functions we have a weaker form of openness, the openness on the hyperbolas $\alpha\beta = c$. To state it it would be convenient to introduce the following slightly artificial definition.
\begin{definition}
We say that a continuous function $A:\R\to \Cm$ belongs to the class $C$ if there exist constants $0 < \delta < \Delta$ and $A_1, A_2\in \R$ such that $A(x) = 0, |x| < \delta$, $A(x) = A_1, x < -\Delta$ and $A(x) = A_2, x > \Delta$.
\end{definition}
We will also need the definition of the Wiener algebra.
\begin{definition}A measurable function $f:\R\to \Cm$ is said to belong to the Wiener algebra $W$ if 
$$||f||_W = \sum_{n\in \Z} \esssup_{n < x < n+1} |f(x)| < \infty.$$
\end{definition}
By $W_0$ we will denote the continuous Wiener algebra, which is $W\cap C(\R)$.

\begin{theorem}\label{main}
Let $A$ be a function from the class $C$, $\gamma \in \R$ and $h\in W_0$. Consider the function $g(x) = A(x)|x|^\gamma + h(x)$. For any $c > 0$ the set of $\alpha > 0$ such that $(\alpha, \frac{c}{\alpha})$ belongs to $\F_g$, is open. 
\end{theorem}

\begin{remark}
The theorem also applies to the situations for which we know only the upper or lower bound in \eqref{frame bound}, in which case the set of $\alpha$ for which it is satisfied is open. Note also that for $\gamma \ge -\frac{1}{2}$ the function $g$ may not even belong to $L^2(\R)$, but the proof works if we consider only the Schwartz test functions $f$.
\end{remark}
Since $M^1\subset W_0$ \cite{feich} (see also \cite[Proposition 12.1.4]{Gro}), this theorem also automatically holds for $h\in M^1$.

By taking $\gamma = -1$ and choosing specific functions $A$ and $h$ we can show that this theorem applies to all rational functions $g$.
\begin{corollary}
Let $g(x) = \frac{P(x)}{Q(x)}$, where $P$ and $Q$ are polynomials with $\deg(P) < \deg(Q)$ and $Q(x) \neq 0, x\in \R$. Then for any $c > 0$ the set of $\alpha > 0$ such that $(\alpha, \frac{c}{\alpha})$ belong to $\F_g$, is open. 
\end{corollary}
\begin{proof}
Let $Q(x) = a_n x^n + \ldots + a_0$, $P(x) = b_{n-1}x^{n-1} + \ldots + b_0$, where $a_n\neq 0$. Take a function $A(x)$ from the class $C$ which is equal to $\frac{b_{n-1}}{a_n}$ for big positive $x$ and which is equal to $-\frac{b_{n-1}}{a_n}$ for big negative $x$. Then, the function $g(x) - A(x)|x|^{-1}$ is continuous, bounded and decays like $O(\frac{1}{x^2})$ for big $x$. Therefore, it belongs to the continuous Wiener algebra and we can apply Theorem \ref{main}.
\end{proof}

In all the examples we have mentioned so far, the non-openness of the frame set came from the points $(\alpha, \beta)$ with $\alpha\beta = 1$. So, it is reasonable to assume that this is due to some "boundary effects" and that the frame set will be open if we restrict our attention to the pairs with $\alpha\beta < 1$. We will now discuss two examples which show that even if we exclude the hyperbola $\alpha\beta = 1$ we cannot guarantee the openness of the frame set.

First of all, if the function $g$ is bounded and supported on the interval $[a, b]$ then it automatically belongs to the Wiener algebra. It is easy to see that if $\alpha > b-a$ then $(\alpha, \beta)$ is not in the frame set simply because the supports of $g_{n, m}$ do not cover the whole $\R$. Similarly, if $\alpha = b-a$ then the system is a frame if and only if $\alpha\beta \le 1$ and $\essinf_{x\in[a,b]} |g(x)| > 0$. So, for the function $g(x) = \chi_{[a, b]}(x)$ points $(b-a, \beta)$ lie in $\F_g$ for all $\beta \le \frac{1}{b-a}$ but no point with $\alpha > b-a$ belongs to the frame set. In general, the frame set for the characteristic function of an interval is a very complicated picture known as Janssen's tie \cite{DS, Jans}.

The second example, which was the motivation for the present work, is the function $g(x) = \frac{1}{x+i} + \frac{1}{x+i+1}$. Together with Yurii Belov we proved the following theorem.
\begin{theorem}
Let $\alpha, \beta > 0$ with $\alpha \beta < 1$ and $g(x) = \frac{1}{x+i} + \frac{1}{x+i+1}$. The point $(\alpha, \beta)$ does not belong to $\F_g$ if and only if $\beta \in \N$ or there exist numbers $n, m\in \N$ such that $\frac{n}{\alpha} = (n-1)\beta + m$ and $\beta \ge [\frac{m}{2}] + \frac{1}{2}$, and if $\gcd(m, 2n) = 1$ then the last inequality is strict.
\end{theorem}
Every pair of numbers $n, m\in \N$ such that $\gcd (m, 2n) = 1$ gives us a point at which the intersection of the frame set and $\{(\alpha, \beta): \alpha\beta < 1\}$ is not open, thus there are countably many points at which the openness fails. In particular, the point $(1, \beta)$ belongs to $\F_g$ if and only if $\beta \le \frac{1}{2}$, so the point $(1, \frac{1}{2})$ is in $\F_g$, but there are points which are arbitrarily close to it which are not $\F_g$. Moreover, we showed that this is the only rational function of degree $2$ up to dilations and complex shifts for which the intersection of the frame set and $\{(\alpha, \beta): \alpha\beta < 1\}$ is not open.
\section{Proof of Theorem \ref{main}}
Let function $g$ be as in the statement of Theorem \ref{main} and assume that $(\alpha, \beta)$ is in $\F_g$. We want to show that for numbers $s$ sufficiently close to $1$ the points $(s\alpha, \frac{\beta}{s})$ are also in $\F_g$. The key idea of the proof is as follows: the point $(s\alpha, \frac{\beta}{s})$ is in the frame set for $g(x)$ if and only if $(\alpha, \beta)$ is in the frame set for the function $g(sx)$. Clearly this is the same as $(\alpha, \beta)$ being in the frame set of $s^{-\gamma}g(sx)$ (we just rescaled the function, so we only have to rescale the constants $A$ and $B$ in \eqref{frame bound}). Note that here $\gamma$ is from the statement of Theorem \ref{main}.

We are going to show that the frame constants for $g(x)$ and $s^{-\gamma}g(sx)$ with respect to the lattice $\alpha\Z\times \beta\Z$ are close if $s$ is close to $1$, thereby showing that if they are positive for $g(x)$ then they are positive for  $s^{-\gamma}g(sx)$ if $s$ is close enough to $1$. Put $g^s(x) = g(x) - s^{-\gamma}g(sx)$. Let $f$ be an arbitrary function from $L^2(\R)$ and consider the following expression:
$$U_s(f) = \sum_{n, m\in\Z} |\langle f, g^s_{n, m}\rangle|^2,$$
where $g^s_{n, m} = g^s(x-\alpha n)e^{2\pi i x \beta m}$. If we know that this quantity is bounded by $C||f||^2$ for some $0 < C < A$ and all $f\in L^2(\R)$, where $A$ is the lower frame bound for $g$, then by the triangle inequality in $\ell^2(\Z^2)$ we would have that the lower frame bound for $s^{-\gamma}g(sx)$ is at least $(\sqrt{A}-\sqrt{C})^2$ while the upper frame bound for $s^{-\gamma}g(sx)$ is at most $(\sqrt{B}+\sqrt{C})^2$, where $B$ is the upper frame bound for $g$. In particular, $(\alpha, \beta)$ would be in the frame set for $s^{-\gamma}g(sx)$. So, it remains to estimate $U_s(f)$. To do so we need the following well-known estimate for the Gabor sums.
\begin{lemma}[\hspace{1sp}{\cite[Corollary 6.2.3]{Gro}}]\label{W lemma}
Let $G\in W$, $\alpha, \beta > 0$. There exists $c = c(\alpha, \beta)$ such that for all $f\in L^2(\R)$ we have
$$\sum_{n, m\in\Z} |\langle f, G_{n, m}\rangle|^2 \le c ||f||_{L^2}^2 ||G||_W^2,$$
where $G_{n, m} = G(x-\alpha n)e^{2\pi i x \beta m}$.
\end{lemma}
With this lemma at hand, to estimate $U_s(f)$ all we need to do is to estimate $||g^s||_W$. We are going to show that $||g^s||_W \to 0$ as $s\to 1$, which would clearly be enough to get the desired estimate. We have
$$g^s(x) = (h(x) - s^{-\gamma}h(sx)) + (A(x)-A(sx))|x|^\gamma,$$
where crucially $s^{-\gamma}$ cancels in the second term. We are going to show that the $W$-norm of both of these terms tends to $0$ as $s$ tends to $1$. We start with the (simpler) second term.

Without loss of generality we can assume that $\frac{1}{2} < s < 2$. If $|x| < \frac{\delta}{2}$ or $|x| > 2\Delta$ then $A(x) = A(sx)$, where $\delta$ and $\Delta$ are from the definition of the function $A$ belonging to the class $C$. For $\frac{\delta}{2} < |x| < 2\Delta$ the value of $|x|^\gamma$ is uniformly bounded, so we can just bound $||A(x)-A(sx)||_W$.

The intervals $\frac{\delta}{2} \le x \le 2\Delta$ and $-2\Delta \le x \le -\frac{\delta}{2}$ can be covered by finitely many intervals of the form $[n, n+1], n\in \Z$, and we can assume that all the points in all of these intervals are at most $N$ in absolute value. For such $x$ the difference between $x$ and $sx$ is at most $|s-1|N$, while both $x$ and $sx$ lie in the interval $[-2N, 2N]$ since $s < 2$. Since $A$ is continuous, it is uniformly continuous on the compact interval $[-2N, 2N]$ so for every $\eps > 0$ there exists $\delta > 0$ such that if $|s-1|N$ is less than $\delta$ then $|A(x)-A(sx)|$ is less than $\eps$. In that case $||A(x)-A(sx)||_W \le (4N+1)\eps$. Since $\eps$ is arbitrary, $||A(x)-A(sx)||_W\to 0$ as $s\to 1$, as required.

We now turn to bounding $|| h(x) - s^{-\gamma}h(sx)||_W$. First, we split this as $$h(x)(1-s^{-\gamma}) + s^{-\gamma}(h(x)-h(sx)).$$

The norm of the first term is $||h||_W(1-s^{-\gamma})$ which tends to $0$ as $s$ tends to $1$. For the second term, if $\frac{1}{2} \le s \le 2$ then $s^{-\gamma}$ is at most $2^{|\gamma|}$. So, it remains to bound $||h(x)-h(sx)||_W$. We have
\begin{equation}\label{W continuity}
||h(x)-h(sx)||_W = \sum_{n\in \Z} \esssup_{n < x < n+1} |h(x)-h(sx)|.
\end{equation}

Since the function $h$ is in $W$, for any $\eps > 0$ there exists $N$ such that
$$\sum_{|n| > N} \esssup_{n < x < n+1} |h(x)| < \eps.$$

We are going to split the sum in \eqref{W continuity} into the terms with $|n| \le 4N+4$ and $|n| > 4N+4$. For the first sum by the argument with uniform continuity similar to the estimate for $||A(x)-A(sx)||_W$ we can see that it tends to $0$ as $s$ tends to $1$. Note that here it is important that $h$ is continuous, that is that $h$ is in $W_0$ and not just in  $W$. For the second sum, we will simply bound $|h(x)-h(sx)| \le |h(x)|+|h(sx)|$ and so it is enough to estimate
$$\sum_{|n|>4N+4} \esssup_{n < x<n+1} |h(x)| + \sum_{|n|>4N+4}\esssup_{n < x < n+1} |h(sx)|.$$

The first sum we already know is at most $\eps$. For the second sum, since $\frac{1}{2} < s < 2$, we can bound each singular essential supremum by a sum of at most three essential supremums from the first sum, and each essential supremum from the first sum will be used at most $3$ times to bound a term in the second sum. Hence, the second sum is at most $9\eps$ and the whole expression is at most $10\eps$. Since $\eps > 0$ was arbitrary we get that $||h(x)-h(sx)||_W\to 0$ as $s\to 1$.
\begin{remark}
The only two places where we used the fact that $h\in W_0$ are the estimate in Lemma \ref{W lemma} and the fact that $||h(x)-h(sx)||_W\to 0$ as $s\to 1$. So, we can replace $W_0$ in the statement of the Theorem \ref{main} by any other space of functions satisfying both of these conditions. In particular, it holds if $h\in \widehat{W_0}$, that is if $\hat{h}\in W_0$.
\end{remark}
\subsection*{Acknowledgments} I would like to thank Yurii Belov and Yurii Lyubarskii for helpful discussions and for encouraging me to write this note. I also would like to thank Hans Georg Feichtinger and Ilya Zlotnikov for the helpful comments on the manuscript. This work was supported by BSF Grant 2020019, ISF Grant 1288/21, and by The Raymond and Beverly Sackler Post-Doctoral Scholarship. 



\end{document}